\documentclass[11pt,leqno]{article}

\usepackage{amsfonts,amsmath,amssymb,amsthm}
\usepackage{fullpage}
\usepackage{dsfont}
\usepackage{hyperref}
\hypersetup{hidelinks}

\newtheorem{theorem}{Theorem}[section]
\newtheorem{proposition}[theorem]{Proposition}
\newtheorem{lemma}[theorem]{Lemma}
\newtheorem{remark}[theorem]{Remark}
\newtheorem{definition}[theorem]{Definition}
\newtheorem{corollary}[theorem]{Corollary}
\numberwithin{equation}{section}

\newcommand{\E}{\mathbb{E}}

\usepackage{dsfont}
\newcommand{\vol}{\operatorname{vol}}
\newcommand{\Var}{\operatorname{Var}}

\usepackage{setspace}
\begin{document}
\small

\title{\bf Cram\'er transform, half-space depth and threshold phenomena for convex bodies}

\author{Minas Pafis}

\date{}

\maketitle

\begin{abstract}\footnotesize
We study the relationship between the Cram\'er transform and Tukey's half-space depth for log-concave probability measures. For the uniform probability measure $\mu_K$ on a convex body $K\subseteq\mathbb{R}^n$, we prove the sharp pointwise comparison
$$\Lambda_K^*(x)\leq -\log q_K(x)\leq \Lambda_K^*(x)+\frac12\log n+C,\qquad x\in\operatorname{int}(K),$$
where $C$ is an absolute constant. The order $\log n$ is optimal, as shown by the Euclidean ball. The proof combines exponential tilting, one-dimensional log-concavity, and self-concordance of the Cram\'er transform.

As consequences, we obtain sharp-order moment and tail estimates for $\Lambda_K^*$ and identify $\exp(\Lambda_K^*(x))$, up to polynomial factors in the dimension, with the number of independent samples needed for $x$ to be captured by their random convex hull. We also establish an $O(n^2)$ variance bound for the logarithmic half-space depth and use it to derive a general criterion for sharp thresholds of random convex hulls. In particular, this criterion applies to the uniform measures on $\ell_p$-balls for every $p>1$.

These results establish a quantitative link between large-deviation cost, geometric depth, and sampling complexity in high-dimensional convex geometry.
\end{abstract}

\section{Introduction}\label{sec:introduction}

Let $\mu$ be a log-concave probability measure on $\mathbb{R}^n$ and let
\[
\Lambda_\mu(\theta)
=
\log\int_{\mathbb{R}^n}e^{\langle\theta,x\rangle}\,d\mu(x),
\qquad \theta\in\mathbb{R}^n,
\]
denote its logarithmic Laplace transform. Its Legendre transform
\[
\Lambda_\mu^*(x)
=
\sup_{\theta\in\mathbb{R}^n}
\left\{
\langle\theta,x\rangle-\Lambda_\mu(\theta)
\right\}
\]
is the Cram\'er transform of $\mu$. It is a fundamental object in
large deviations and also plays an important role in high-dimensional
convex geometry, where it is related to exponential families,
entropic barriers, and concentration phenomena.

A different notion of centrality is provided by Tukey's half-space
depth. For a probability measure $\mu$ on $\mathbb{R}^n$ and
$x\in\mathbb{R}^n$, define
\[
q_\mu(x)
=
\inf\left\{
\mu(H): H\subseteq\mathbb{R}^n
\text{ is a closed half-space containing }x
\right\}.
\]
We refer the reader to the survey article of Nagy, Sch\"{u}tt and Werner
\cite{Nagy-Schutt-Werner-2019} for an extensive and comprehensive survey on Tukey's half-space depth, with an emphasis on its connections
with convex geometry, and many references. 

Thus $q_\mu(x)$ measures the geometric depth of $x$, whereas
$\Lambda_\mu^*(x)$ measures the exponential cost associated with
reaching $x$. The purpose of this paper is to investigate the
relationship between these two notions of outlyingness and to explore
its consequences for random convex hulls.

The main result concerns the uniform probability measure on a convex
body. Let $K\subseteq\mathbb{R}^n$ be a convex body, let $\mu_K$ denote
the uniform probability measure on $K$, and write
\[
\Lambda_K^*=\Lambda_{\mu_K}^*,
\qquad
q_K=q_{\mu_K}.
\]
We prove the following sharp comparison.

\begin{theorem}\label{thm:intro-main-comparison}
There exists an absolute constant $C>0$ such that, for every convex
body $K\subseteq\mathbb{R}^n$ and every $x\in\operatorname{int}(K)$,
\[
q_K(x)
\geq
\frac{C}{\sqrt n}\,e^{-\Lambda_K^*(x)}.
\]
Equivalently,
\[
\Lambda_K^*(x)
\leq
-\log q_K(x)
\leq
\Lambda_K^*(x)+\frac12\log n+C.
\]
\end{theorem}

The reverse inequality
\[
q_K(x)\leq e^{-\Lambda_K^*(x)}
\]
is elementary and holds for every $x\in\operatorname{int}(K)$.
Thus the substantive part of Theorem~\ref{thm:intro-main-comparison}
is the lower bound on $q_K$. It improves the previously known
$O(\sqrt n)$ error in the logarithmic comparison to the optimal
$O(\log n)$ scale. This order cannot, in general, be improved: for
every fixed $r\in(0,1)$,
\[
q_{B_2^n}(re_1)
\asymp_r
\frac{1}{\sqrt n}e^{-\Lambda_{B_2^n}^*(re_1)},
\]
and consequently
\[
-\log q_{B_2^n}(re_1)
=
\Lambda_{B_2^n}^*(re_1)
+\frac12\log n+O_r(1).
\]

The proof of Theorem~\ref{thm:intro-main-comparison} combines
exponential tilting with a one-dimensional reduction and the
self-concordance structure of the Cram\'er transform. The
key estimate is the optimal bound
\[
t^2\Lambda_Y''(t)\leq n
\]
that comes from the entropic barrier, together with a local
lower bound for centered one-dimensional log-concave densities. This
produces the single factor $n^{-1/2}$ in the comparison.

The bounded-support setting allows for a substantially sharper dimensional statement than in the general log-concave setting. Brazitikos and Chasapis \cite{Brazitikos-Chasapis-2024}
proved that for every log-concave probability measure $\mu$ and every
$\varepsilon\in(0,1)$,
\begin{equation}\label{eq:dimension-free}
\Lambda_\mu^*(x)
\geq
(1-\varepsilon)\log\frac{1}{q_\mu(x)}
+
\log\frac{\varepsilon}{2^{1-\varepsilon}}.
\end{equation}
Their result applies to arbitrary log-concave measures and is dimension-free, whereas
Theorem~\ref{thm:intro-main-comparison} exploits the geometry of convex
bodies to obtain an additive $O(\log n)$ comparison.

\medskip

The comparison has a direct interpretation in terms of random convex
hulls. Let $X_1,X_2,\ldots$ be independent random vectors with law
$\mu$ and set
\[
K_N=\operatorname{conv}\{X_1,\ldots,X_N\}.
\]
For $x\in\mathbb{R}^n$, define the median sampling threshold
\[
N_\mu(x)
=
\inf\left\{
N\in\mathbb{N}:
\mathbb{P}(x\in K_N)\geq\frac12
\right\}.
\]
A theorem of Hayakawa, Lyons and Oberhauser \cite{HLO} gives
\[
\frac12\leq N_\mu(x)q_\mu(x)\leq 3n+1.
\]
Consequently, Theorem~\ref{thm:intro-main-comparison} yields the
following.

\begin{theorem}\label{thm:intro-sampling}
There exists an absolute constant $C>0$ such that, for every convex
body $K\subseteq\mathbb{R}^n$ and every $x\in\operatorname{int}(K)$,
\[
\frac12e^{\Lambda_K^*(x)}
\leq
N_K(x)
\leq
Cn^{3/2}e^{\Lambda_K^*(x)}.
\]
\end{theorem}

In particular, the Cram\'er transform describes the logarithmic
sampling complexity of capturing a point by a random convex hull:
up to polynomial factors in the dimension,
\[
N_K(x)\asymp e^{\Lambda_K^*(x)}.
\]
This gives a probabilistic interpretation of $\Lambda_K^*$ that will
be useful throughout the paper.

The Cram\'er transform can also be large well inside a convex body.
For $\eta\in(0,1)$, define
\[
M_\eta(K)
=
\max_{x\in\eta K}\Lambda_K^*(x).
\]
We prove the following.

\begin{theorem}\label{thm:intro-interior}
There exist absolute constants $c,C>0$ such that, for every centered
convex body $K\subseteq\mathbb{R}^n$ and every $\eta\in(0,1)$,
\[
c\eta^2n
\leq
M_\eta(K)
\leq
n\log\frac{1}{1-\eta}.
\]
\end{theorem}

Thus, for every fixed $\eta\in(0,1)$, there are points in $\eta K$ whose Cram\'er transform is of order $n$. By Theorem~\ref{thm:intro-sampling}, such points require exponentially many random samples to be captured with probability bounded away from zero. The upper bound follows from
\[
x+(1-\eta)K\subseteq K,
\qquad x\in\eta K,
\]
while the lower bound reflects the geometry of Cram\'er sublevel sets and its self-concordance property.

\medskip

We next turn to the distribution of the Cram\'er transform under the
measure itself. For convex bodies we prove moment estimates of the
form
\[
\|\Lambda_K^*\|_{L^p(\mu_K)}
\lesssim n(\log n+p),
\qquad p\geq1.
\]
More generally, the same order holds for arbitrary log-concave
probability measures.

\begin{theorem}\label{thm:intro-moments-logconcave}
There exists an absolute constant $C>0$ such that, for every
log-concave probability measure $\mu$ on $\mathbb{R}^n$ and every
$p\geq1$,
\[
\|\Lambda_\mu^*\|_{L^p(\mu)}
\leq
Cn(\log n+p).
\]
Equivalently, after changing the absolute constant,
\[
\mu\left(
\left\{
x:\Lambda_\mu^*(x)>Cn(\log n+u)
\right\}
\right)
\leq e^{-u},
\qquad u\geq0.
\]
\end{theorem}
The Euclidean ball shows that both contributions are necessary. Indeed,
$$\|\Lambda_{B_2^n}^*\|_{L^p(\mu_{B_2^n})}\asymp n(\log n+p),\qquad p\geq1.$$

For convex bodies we obtain in addition the exponential-moment
estimate
\[
\mathbb{E}_{\mu_K}
\exp\left(\frac{a}{n}\Lambda_K^*\right)
\leq
2^a\Gamma(1-a)n^a,
\qquad a\in(0,1).
\]
These estimates complement
the recent work of Giannopoulos and Tziotziou
\cite{Giannopoulos-Tziotziou-2025}, who established exponential
integrability of the Cram\'er transform at scale $1/n$ for centered
log-concave probability measures and, in particular, obtained the
optimal-order estimate
\[
\|\Lambda_\mu^*\|_{L^2(\mu)}
\lesssim n\log n.
\]
Our result makes the dependence on the moment parameter explicit.

\medskip

A second ingredient in our analysis is the distribution of
half-space depth. Here the relevant structure is $s$-concavity.
Recall that a probability measure $\mu$ is $s$-concave,
$s\in(0,1/n]$, if
\[
\mu((1-\lambda)A+\lambda B)^s
\geq
(1-\lambda)\mu(A)^s+\lambda\mu(B)^s
\]
for all non-empty compact sets $A,B$ and $\lambda\in[0,1]$.
The uniform probability measure on an $n$-dimensional convex body is
$1/n$-concave.

The key observation is that, for an $s$-concave measure supported on
a convex body, the function $q_\mu^s$ is concave. An application of generalized Berwald's inequality then gives a dimension-free estimate at the
natural $s$-concavity scale.

\begin{theorem}\label{thm:intro-depth-variance}
Let $\mu$ be an $s$-concave probability measure supported on a convex
body $K$. Then
\[
\operatorname{Var}_\mu(-\log q_\mu)
\leq
\frac{\pi^2}{6s^2}.
\]
In particular, for the uniform probability measure on a convex body,
\[
\operatorname{Var}_{\mu_K}(-\log q_K)
\leq
\frac{\pi^2}{6}n^2.
\]
\end{theorem}

This estimate is sharp in order for uniform measures, as illustrated
by the Euclidean ball. Combined with
Theorem~\ref{thm:intro-main-comparison}, it shows that the Cram\'er
transform and logarithmic half-space depth have fluctuations of order
at most $n$, while their means may be much larger.

This observation leads to a general sharp-threshold criterion for
random convex hulls. The precise meaning of this notion is given in Definition~\ref{def:sharp-threshold} below. Roughly speaking, if the mean Cram\'er transform
is large compared with $n$, then the logarithmic sample size required
for a random convex hull to capture a typical point is concentrated
around that mean. 

\begin{theorem}\label{thm:intro-sharp-threshold}
Let $(K_n)_{n\in\mathbb{N}}$ be a sequence of convex bodies in
$\mathbb{R}^n$, and let $\mu_{K_n}$ denote the corresponding uniform
probability measures. If
\[
\frac{1}{n}
\mathbb{E}_{\mu_{K_n}}
\left[\Lambda_{K_n}^*\right]
\longrightarrow+\infty,
\]
then the sequence $(\mu_{K_n})_{n\in\mathbb{N}}$ exhibits a sharp
threshold for random convex hulls around
\[
\left(
\mathbb{E}_{\mu_{K_n}}[\Lambda_{K_n}^*]
\right)_{n\in\mathbb{N}}.
\]
\end{theorem}

The same criterion can be expressed in terms of logarithmic
half-space depth: if
\[
\frac{1}{n}
\mathbb{E}_{\mu_{K_n}}
[-\log q_{K_n}]
\longrightarrow+\infty,
\]
then the corresponding random convex hulls have a sharp threshold.

As an application, we consider the $\ell_p$-balls.

\begin{theorem}\label{thm:intro-pballs}
Let $p>1$. There exist constants $C_p,c_p>0$, depending only on $p$,
such that
\[
\mathbb{E}_{\mu_{B_p^n}}
[-\log q_{B_p^n}]
\geq
C_p n\log n-c_p n.
\]
Consequently, the sequence of uniform measures on the $\ell_p$-balls
exhibits a sharp threshold for random convex hulls.
\end{theorem}

The proof uses estimates for caps of $B_p^n$. For $1<p\leq2$,
uniform convexity yields caps of diameter of order $\sqrt h$, whereas
for $p>2$, Clarkson's inequality gives diameter of order $h^{1/p}$.
These estimates imply the required logarithmic growth of the expected
logarithmic depth.

\medskip

\noindent\textbf{Relation to previous work.}

The connection between random convex hulls and half-space depth has a
long history. Tukey's half-space depth is a classical notion of
statistical depth, while threshold phenomena for random polytopes
have been studied extensively in convex geometry and probability. In
particular, Dyer, F\"uredi and McDiarmid \cite{DFM} established sharp
threshold phenomena for random polytopes in the cube, followed by
work for various product and log-concave models; see e.g. \cite{Bonnet-Chasapis-Grote-Temesvari-Turchi-2019,BKT,Brazitikos-Chasapis-2024,BGP-depth,BGP-threshold,BP-discrete,CTV,FPT,GaGia,Giannopoulos-Tziotziou-2025,Pafis,Pivov}.

The Cram\'er transform as a natural parameter for random-convex-hull
thresholds was developed by Brazitikos, Giannopoulos and Pafis in
\cite{BGP-depth} and \cite{BGP-threshold}. They established
quantitative comparisons between the Cram\'er transform and
half-space depth for uniform measures on convex bodies, with an
$O(\sqrt n)$ error in the logarithmic comparison. Theorem
\ref{thm:intro-main-comparison} improves this to the optimal
$O(\log n)$ scale.

As mentioned above, Brazitikos and Chasapis
\cite{Brazitikos-Chasapis-2024} obtained a dimension-free comparison
for arbitrary log-concave probability measures. The two results are
complementary: their theorem has greater scope, while ours gives a
sharper additive dimensional error in the bounded-support setting.

The moment estimates for the Cram\'er transform are related to the
recent work of Giannopoulos and Tziotziou
\cite{Giannopoulos-Tziotziou-2025}. Our contribution is the explicit
two-parameter estimate
\[
\|\Lambda_\mu^*\|_{L^p(\mu)}
\lesssim n(\log n+p),
\]
together with the corresponding tail bound and, for convex bodies,
a more precise exponential-moment estimate.

Finally, Hayakawa, Lyons and Oberhauser \cite{HLO} proved the general
estimate
\[
\frac12\leq N_\mu(x)q_\mu(x)\leq3n+1.
\]
This result provides the combinatorial link between half-space depth
and random convex hulls. Combined with our pointwise comparison, it
turns the Cram\'er transform into a direct measure of logarithmic
sampling complexity.

\medskip

The results suggest a common framework connecting three notions of
complexity. The Cram\'er transform measures the analytic cost of
reaching a point, half-space depth measures its geometric centrality,
and random convex hulls convert these quantities into sampling
thresholds. For uniform measures on convex bodies, the first two
quantities agree up to the optimal $O(\log n)$ additive error, while
the $s$-concavity structure controls their fluctuations. This allows
estimates for one quantity to be transferred systematically to the
others.

The paper is organized as follows. Section~\ref{sec:notation} introduces the notation used throughout the paper. In
Section~\ref{sec:self-concordance} we establish the
$1/s$-self-concordance property for $s$-concave probability measures.
In Section~\ref{sec:depth-cramer} we prove the pointwise comparison
between the Cram\'er transform and half-space depth and establish its
sharpness for Euclidean balls. Section~\ref{sec:cramer-inside-convex-bodies}
studies the size of the Cram\'er transform inside convex bodies.
Section~\ref{sec:moments-cramer-transform} develops the moment and
exponential-integrability estimates. Sections
\ref{sec:s-concavity-half-space-depth} and
\ref{sec:threshold-random-convex-hulls} develop the
$s$-concavity and random-convex-hull arguments, including the sharp
threshold criterion and its application to $\ell_p$-balls.


\section{Notation and backround information}\label{sec:notation}

First, we introduce some basic notation and definitions. We work in ${\mathbb R}^n$, which is equipped with the standard inner product  $\langle\cdot ,\cdot\rangle $. Volume in $\mathbb{R}^n$ is denoted by $\vol_n$. For $p \geq 1$ we denote by $\|\cdot\|_p$ the $\ell_p$-norm and by $B_p^n$ the corresponding unit ball. We write $S^{n-1}$ for the Euclidean unit sphere in $\mathbb{R}^n$, i.e. \[S^{n-1}=\{x\in \mathbb{R}^n: \|x\|_2=1\}.\]

The letters $C, c, c_1, c_2$ etc. denote absolute positive constants whose value may change from line to line. Whenever we
write $a\asymp b$, we mean that there exist absolute constants $c_1,c_2>0$ such that $c_1a\leq b\leq c_2a$. Also the notation $a\asymp_\ell b$ means that the implicit constants depend additionally on $\ell$.

A convex body in ${\mathbb R}^n$ is a compact convex subset $K$ of ${\mathbb R}^n$ with non-empty interior.
We say that $K$ is centered if its barycenter ${\rm bar}(K)$ is at the origin, i.e. if
\begin{equation*}\int_K\langle x,u\rangle \, dx=0\end{equation*} for every $u\in S^{n-1}$.
The support function of $K$ is defined for every $y\in {\mathbb R}^n$ by
$h_K(y)=\max \{\langle x,y\rangle :x\in K\}$. If $0 \in \operatorname{int} (K)$,  we define the Minkowski functional of $K$, as $\|y\|_K=\inf \{ t>0: y \in tK\}$ for every $y \in \mathbb{R}^n$.

We say that a Borel probability measure $\mu $ on $\mathbb{R}^n$ is symmetric if $\mu (-B)=\mu (B)$ for every Borel subset $B$
of ${\mathbb R}^n$ and that $\mu $ is centered if the barycenter ${\rm bar}(\mu)=\int_{\mathbb{R}^n}x\,d\mu(x)$ of $\mu$ is at the origin, i.e.
\begin{equation*}\int_{\mathbb R^n} \langle x, u \rangle d\mu(x)= 0\end{equation*}
for all $u\in S^{n-1}$. Moreover, we say that $\mu$ is full-dimensional if $\mu(H)<1$ for every
hyperplane $H$ in ${\mathbb R}^n$.

A Borel measure $\mu$ on $\mathbb R^n$ is called log-concave if it is full-dimensional and
$$\mu(\lambda A+(1-\lambda)B) \geq \mu(A)^{\lambda}\mu(B)^{1-\lambda}$$
for any pair of compact sets $A,B$ in ${\mathbb R}^n$ and any $\lambda \in [0,1]$. Borell \cite{Borell-1974} has proved that, under these assumptions, $\mu $ has a log-concave density $f_{{\mu }}$. Recall that a function $f:\mathbb R^n \rightarrow [0,\infty)$ is called log-concave if its support $\{f>0\}$ is a convex set in ${\mathbb R}^n$ and the restriction of $\ln{f}$ to it is concave. 

If $\mu$ be a log-concave probability measure on $\mathbb R^n$, then for any $p\geq 1$ we define the
$L_p$-centroid body $Z_p(\mu)$ of $\mu $ as the convex body
whose support function is
\begin{equation*} h_{Z_p(\mu)}(y):=\left(
\int_{\mathbb R^n} |\langle x,y\rangle|^p f_{\mu}(x)dx \right)^{1/p},\qquad y\in {\mathbb R}^n.
\end{equation*} For $p \geq 1$ we also consider the convex bodies $Z_p^+(\mu)$, with support function \begin{equation*}h_{Z_p^+(\mu )}(y)=\left (\int_{{\mathbb R}^n}\langle
x,y\rangle_+^pf_{\mu }(x)dx\right )^{1/p},\qquad y\in {\mathbb R}^n,\end{equation*} where $a_+=\max\{a,0\}$. Observe that \[Z_p^+(\mu) \subseteq Z_p(\mu), \qquad p \geq 1.\]

We say that a measure $\mu $ on
${\mathbb R}^n$ is $s$-concave for some $0< s \leq 1/n$, if it is full-dimensional and
\begin{equation*}\mu ((1-\lambda )A+\lambda B)^s\geq
(1-\lambda )\mu^{s}(A)+\lambda \mu^{s}(B)\end{equation*}
for any pair of compact sets $A,B$ in ${\mathbb R}^n$  and any $\lambda\in [0,1]$. Every $s$-concave measure is compactly supported and is also log-concave. Moreover, if $\mu $ is $s$-concave and $0<s^{\prime}\leq s$ then $\mu $ is also $s^{\prime}$-concave. The Brunn-Minkowski inequality implies that if $\mu_K$ the uniform probability measure on a convex body $K$ in $\mathbb{R}^n$, then $\mu_K$ is $1/n$-concave.

A function $f:\mathbb R^n\to [0,\infty)$ is called $\gamma $-concave for some $\gamma>0$ if
\begin{equation*}f((1-\lambda )x+\lambda y)^\gamma\geq (1-\lambda )f^{\gamma }(x)+\lambda f^{\gamma }(y)\end{equation*}
for all $x,y\in {\mathbb R}^n$ with $f(x)f(y)>0$ and all $\lambda\in [0,1]$. Borell \cite{Borell-1975}
showed that if $\mu$ is a full-dimensional measure on ${\mathbb R}^n$ then for every $0<  s <1/n$ we have that
$\mu $ is $s$-concave if and only if it has a non-negative density $f\in L_{{\rm loc}}^1({\mathbb R}^n,dx)$
which is $\frac{s}{1-sn}$-concave.


\section{Self-concordance}\label{sec:self-concordance}

\begin{definition}\label{def:self-conc}\rm Let $K$ be a convex body in $\mathbb{R}^n$ and let $\Phi:\operatorname{int}(K)\to\mathbb{R}$. We say that $\Phi$ is a
barrier for $K$ if
$$\Phi(x)\longrightarrow+\infty,\qquad\text{as }\,x\longrightarrow\partial K.$$
If $\Phi$ is $C^3$, we say that it is self-concordant if, for every $x\in\operatorname{int}(K)$ and every $v\in\mathbb{R}^n$,
$$\left|\nabla^3\Phi(x)[v,v,v]\right|\leq 2\left\langle\nabla^2\Phi(x)v,v\right\rangle^{3/2}.$$
Furthermore, $\Phi$ is called $\alpha$-self-concordant if, in addition, for every $x\in\operatorname{int}(K)$ and every $v\in\mathbb{R}^n$,
$$\langle\nabla\Phi(x),v\rangle^2\leq\alpha\left\langle\nabla^2\Phi(x)v,v\right\rangle.$$
\end{definition}

Self-concordance is a central notion in the theory of interior-point methods and convex optimization, see e.g. \cite{NesterovNemirovskii}. In the context of convex geometry, Bubeck and Eldan \cite{BubeckEldan} showed that the Cram\'{e}r transform of the uniform measure on a convex body is a $(1+o(1))n$-self-concordant barrier. This bound was subsequently sharpened to the optimal parameter $n$ by Chewi \cite{n-self}. While these results concern the uniform measure on a convex body, we show that the self-concordance phenomenon extends naturally from uniform measures on convex bodies to the broader class of
$s$-concave measures and gives a sharp dependence on the concavity parameter. We expect that this result is well known to experts in the field; nevertheless, we include a proof for completeness. Our proof is a variation of a technique used in \cite{Fradelizi-Madiman-Wang-2016}.

\begin{theorem}\label{thm:s-concave-self-concordance}Let $0<s\leq 1/n$, and let $\mu$ be an $s$-concave probability measure
supported on a convex body $K$. Then $\Lambda_\mu^*$ is a $1/s$-self-concordant barrier for $K$.
\end{theorem}

\begin{proof}The asymptotic behavior of $\Lambda_\mu^*$ near the boundary of $K$ is immediate. The self-concordance of $\Lambda_\mu^*$ follows from
\cite[Lemma~1, Lemma~2]{BubeckEldan}, since the exponential tilts \[d\mu_\theta(x)=e^{\langle\theta,x\rangle-\Lambda_{\mu}(\theta)}\,d\mu(x), \qquad \theta \in \mathbb{R}^n\] of an $s$-concave measure are log-concave. It remains to prove the $1/s$-self-concordance inequality.

We first establish the required estimate for the Hessian of $\Lambda_\mu^*$. Fix $u\in S^{n-1}$ and set
$$Y=\langle u,X\rangle,\quad b=h_K(u),$$
where $h_K$ denotes the support function of $K$. Define
$$Z=b-Y.$$
Then $Z$ is supported on $[0,\infty)$.

Let $h$ denote the density of $Z$. Since $\mu$ is $s$-concave, its one-dimensional marginals are also $s$-concave. In particular, $h$ can
be written in the form
$$h(z)=\psi(z)^r,\qquad r=\frac{1}{s}-1,$$
where $\psi$ is a nonnegative concave function on its support.

For $t>0$, consider the Laplace transform
$$I(t)=\mathbb{E}[e^{-tZ}]=\int_0^\infty e^{-tz}\psi(z)^r\,dz.$$
After the change of variables $w=tz$, we obtain
$$t^{1/s}I(t)=\int_0^\infty e^{-w}\left[t\psi\left(\frac{w}{t}\right)\right]^r\,dw.$$

The map
$$(t,w)\longmapsto t\psi\left(\frac{w}{t}\right)$$
is the perspective of the concave function $\psi$ and is therefore concave on its natural domain. Since $r>0$, the function
$$(t,w)\longmapsto e^{-w}\left[t\psi\left(\frac{w}{t}\right)\right]^r$$
is log-concave. By the Pr\'ekopa--Leindler theorem, its integral with respect to $w$ is log-concave as a function of $t$. Hence
$$t\longmapsto t^{1/s}I(t)$$
is log-concave.

Consequently,
$$\frac{d^2}{dt^2}\log\left(t^{1/s}I(t)\right)\leq 0,$$
and therefore
$$\frac{d^2}{dt^2}\log I(t)\leq\frac{1}{st^2}.$$
On the other hand,
$$I(t)=\mathbb{E}[e^{-t(b-Y)}]=e^{-tb+\Lambda_\mu(tu)}.$$
Thus
$$\frac{d^2}{dt^2}\log I(t)=\frac{d^2}{dt^2}\Lambda_\mu(tu)=\operatorname{Var}_{\mu_{tu}}\bigl(\langle X,u\rangle\bigr),$$
where
$$d\mu_{tu}(x)=e^{t\langle u,x\rangle-\Lambda_\mu(tu)}\,d\mu(x)$$
is the exponential tilt of $\mu$ and $X$ is distributed according to $\mu_{tu}$.

It follows that
$$\operatorname{Var}_{\mu_{tu}}\bigl(\langle X,u\rangle\bigr)\leq\frac{1}{st^2}.$$
Equivalently,
\begin{equation}\label{eq:tilted-variance-upper-bound}
\operatorname{Var}_{\mu_{tu}}\bigl(\langle X,tu\rangle\bigr)\leq\frac{1}{s}.\end{equation}
We now transfer this estimate to the Hessian of the Legendre transform. Let
$$x=\nabla\Lambda_\mu(tu).$$
By the standard duality relations between $\Lambda_\mu$ and $\Lambda_\mu^*$,
$$tu=\nabla\Lambda_\mu^*(x)$$
and
$$\nabla^2\Lambda_\mu^*(x)=\bigl(\nabla^2\Lambda_\mu(tu)\bigr)^{-1}.$$
Hence
\begin{align*}
\left\langle\nabla\Lambda_\mu^*(x),\bigl(\nabla^2\Lambda_\mu^*(x)\bigr)^{-1}\nabla\Lambda_\mu^*(x)\right\rangle
&=\left\langle tu,\nabla^2\Lambda_\mu(tu)\,tu\right\rangle\\
&=\operatorname{Var}_{\mu_{tu}}\bigl(\langle X,tu\rangle\bigr)\\
&\leq\frac{1}{s}.
\end{align*}
Thus $\Lambda_\mu^*$ satisfies the gradient--Hessian inequality
$$\left\langle\nabla\Lambda_\mu^*(x),\bigl(\nabla^2\Lambda_\mu^*(x)\bigr)^{-1}\nabla\Lambda_\mu^*(x)\right\rangle\leq\frac{1}{s}.$$
This is equivalent to the self-concordance-parameter inequality in Definition~\ref{def:self-conc}, with parameter $1/s$.
\end{proof}

In fact, if $\Lambda_\mu^*$ is sufficiently large, we can also obtain a lower bound for the quantity on the left-hand side of
\eqref{eq:tilted-variance-upper-bound}. We will first need the following standard consequence of self-concordance; see \cite[Theorem 4.1.7]{book-optim}.

\begin{lemma}\label{lem:self-conc-inequality-local}If $f:\operatorname{int}(K)\to\mathbb{R}$ is self-concordant, then
$$f(y)\geq f(x)+\langle\nabla f(x),y-x\rangle+\omega\left(\sqrt{\left\langle\nabla^2f(x)(y-x),y-x\right\rangle}\right),
\qquad x,y\in\operatorname{int}(K),$$
where
$$\omega(s)=s-\log(1+s),\quad s\geq 0.$$
\end{lemma}

\begin{proposition} Let $\mu$ be a centered $s$-concave probability measure supported on $K$. If $x\in \operatorname{int}(K)$ and
$$\Lambda_\mu^*(x)\geq\log 2,$$
then
$$\frac{1}{2}\leq\operatorname{Var}_{\mu_{tu}}\bigl(\langle X,tu\rangle\bigr)\leq\frac{1}{s},$$
where $x=\nabla\Lambda_\mu(tu)$ for some unique $u\in S^{n-1}$ and $t>0$.
\end{proposition}

\begin{proof}
The upper bound follows from Theorem~\ref{thm:s-concave-self-concordance}. For the lower bound, set
$$\sigma^2(x)=\operatorname{Var}_{\mu_{tu}}\bigl(\langle X,tu\rangle\bigr),
\qquad r=\sqrt{\left\langle\nabla^2\Lambda_\mu^*(x)x,x\right\rangle}.$$
We apply Lemma~\ref{lem:self-conc-inequality-local} with $f=\Lambda_\mu^*$, $y=0$, and the given point $x$. This gives
$$\Lambda_\mu^*(0)\geq\Lambda_\mu^*(x)-\langle\nabla\Lambda_\mu^*(x),x\rangle +\omega(r).$$
Since $\mu$ is centered, $\Lambda_\mu^*(0)=0$. Moreover, by Cauchy--Schwarz,
$$\langle\nabla\Lambda_\mu^*(x),x\rangle\leq\sigma(x)r.$$
Therefore,
$$\Lambda_\mu^*(x)\leq\sigma(x)r-\omega(r).$$
If $\sigma(x)<1/2$, maximizing the right-hand side over $r\geq0$ gives
$$\Lambda_\mu^*(x)\leq -\sigma(x)-\log(1-\sigma(x))<\log 2-\sigma(x)\leq\log 2,$$
which contradicts the assumption $\Lambda_\mu^*(x)\geq\log 2$. Hence
$$\sigma(x)\geq\frac{1}{2},$$
and the proof follows.
\end{proof}

As we will see in Section~\ref{sec:depth-cramer}, self-concordance is not merely an auxiliary analytic property in this setting. It provides the mechanism 
behind the control of the local variance of exponentially tilted measures, which in turn yields the sharp comparison between half-space depth and the Cram\'{e}r transform.

\section{Half-space depth and the Cram\'{e}r transform}\label{sec:depth-cramer}

The purpose of this section is to compare the half-space depth $q_\mu$ with the Cram\'{e}r transform $\Lambda_\mu^*$. We begin with an easy observation. If $\mu$ is a Borel probability measure on $\mathbb{R}^n$ and $T$ is an invertible affine transformation, then for the push-forward probability measure $T_\ast\mu$ we have that $$\Lambda_{T_\ast\mu}^\ast(x)=\Lambda_{\mu}^\ast(T^{-1}x), \quad q_{T_\ast\mu}(x)=q_\mu(T^{-1}x)$$ for every $x \in \mathbb{R}^n$. Therefore, inequalities involving only the Cram\'{e}r transform and the half-space depth, are affinely invariant. The same is true about the expectations of functions of $\Lambda_\mu^*$ and $q_\mu$. 

Let now $\mu$ be an $s$-concave probability measure on $\mathbb{R}^n$, supported on a convex body $K$. The main result is the following pointwise estimate.

\begin{theorem}\label{thm:q-lower-bound-star}
There exists an absolute constant $C>0$ such that
$$q_\mu(x)\geq C\sqrt{s}\,e^{-\Lambda_\mu^*(x)},\qquad x\in\operatorname{int}(K).$$
In particular, for the uniform probability measure on $K$,
$$q_K(x)\geq\frac{C}{\sqrt{n}}e^{-\Lambda_K^*(x)},\qquad x\in\operatorname{int}(K).$$
\end{theorem}

\begin{proof}
We may assume without loss of generality that $\mu$ is centered. 
By Gr\"unbaum's lemma (see \cite[Lemma~2.2.6]{BGVV-book}, it suffices to consider half-spaces of the form
$$H^+=\left\{y\in\mathbb{R}^n:\langle y,u\rangle\geq\langle x,u\rangle\right\},$$
where $u\in S^{n-1}$ and $\langle x,u\rangle>0$. Set
$$h=\langle x,u\rangle$$
and let $t>0$ be the unique solution of
$$\langle\nabla\Lambda_\mu(tu),u\rangle=h.$$
Then
$$\nabla\Lambda_\mu(tu)=x.$$

Let $X$ be a random vector distributed according to the exponential tilt $\mu_{tu}$, and set
$$Z=t\bigl(\langle X,u\rangle-h\bigr).$$
Then $Z$ is centered and log-concave, with variance
$$\sigma_t^2=\operatorname{Var}_{\mu_{tu}}\bigl(\langle X,tu\rangle\bigr)\leq\frac{1}{s}$$
by \eqref{eq:tilted-variance-upper-bound}. Moreover,
$$\mu\left(\left\{y:\langle y,u\rangle\geq h\right\}\right)=e^{-\Lambda_\mu^*(x)}\mathbb{E}_{\mu_{tu}}\left[e^{-Z}\mathbf{1}_{\{Z\geq0\}}\right].$$

We now use a standard one-dimensional fact about log-concave probability densities. There exist absolute constants $c_0,c_1>0$
such that the density $g$ of every centered, variance-one, one-dimensional log-concave random variable satisfies
$$g(z)\geq c_0,\qquad 0\leq z\leq c_1.$$
For a proof, see \cite[Lemma~5.5, Theorem~5.14]{LV07}. After rescaling, the density of $Z$ is bounded below by $c_0/\sigma_t$ on $[0,c_1\sigma_t]$. Hence
\begin{align*}
\mathbb{E}_{\mu_{tu}}\left[e^{-Z}\mathbf{1}_{\{Z\geq0\}}\right]
&\geq \frac{c_0}{\sigma_t}\int_0^{c_1\sigma_t}e^{-z}\,dz\\
&\geq\frac{C}{1+\sigma_t}.
\end{align*}
Since $\sigma_t\leq 1/\sqrt{s}$ and $s \leq 1/n$, this gives
$$\mathbb{E}_{\mu_{tu}}\left[e^{-Z}\mathbf{1}_{\{Z\geq0\}}\right]\geq C\sqrt{s}.$$
Therefore,
$$\mu(H^+)\geq C\sqrt{s}\,e^{-\Lambda_\mu^*(x)}.$$
Taking the infimum over all relevant half-spaces proves the theorem.
\end{proof}

The factor $\sqrt{n}$ for the uniform measure on a convex body in Theorem~\ref{thm:q-lower-bound-star} is optimal. This can already be
seen for Euclidean balls.

\begin{proposition}\label{prop:optimality-of-inequality}
Let $B_2^n$ denote the Euclidean unit ball. For every fixed $r\in(0,1)$,
$$q_{B_2^n}(re_1)\asymp_r\frac{1}{\sqrt{n}} e^{-\Lambda_{B_2^n}^*(re_1)}.$$
\end{proposition}

\begin{proof}
By rotational invariance,
$$\Lambda_{B_2^n}^*(x)=\Lambda_{X_1}^*(\|x\|_2),$$
where $X$ is uniformly distributed on $B_2^n$. If $\theta$ is uniformly distributed on $S^{n+1}$, then
$$X_1\overset{d}{=}\theta_1.$$
Following \cite[Section~4.1]{Brazitikos-Chasapis-2024}, the Laplace transform of $\theta_1$ can be expressed in terms of the modified Bessel function:
$$e^{\Lambda_{\theta_1}(t)}=\Gamma\left(\frac{n}{2}+1\right)\left(\frac{t}{2}\right)^{-n/2}I_{n/2}(t).$$
Differentiating and using the identity
$$I_a'(t)=I_{a+1}(t)+\frac{a}{t}I_a(t)$$
gives
$$\Lambda_{\theta_1}'(t)=R_{n/2}(t),\qquad R_a(t):=\frac{I_{a+1}(t)}{I_a(t)}.$$
We use the Amos-type bounds (see equation (9) in \cite{Amos})
$$G_{a+1}(t)\leq R_a(t)\leq G_a(t),$$
where
$$G_a(t)=\frac{t}{a+\sqrt{t^2+a^2}},\qquad G_a^{-1}(r)=\frac{2ar}{1-r^2}.$$
Let $t_r$ be determined by
$$\Lambda_{\theta_1}'(t_r)=r.$$
Putting $a=n/2$ in the preceding inequalities gives
$$G_{a+1}(t_r)\leq r\leq G_a(t_r).$$
Since $G_a$ is increasing,
$$\frac{nr}{1-r^2}\leq (\Lambda_{\theta_1}^*)'(r)\leq\frac{(n+2)r}{1-r^2}.$$
Integrating from $0$ to $r$ yields
\begin{equation}\label{eq:unit-ball-star-bounds}
-\frac{n}{2}\log(1-r^2)\leq\Lambda_{B_2^n}^*(re_1)\leq -\frac{n+2}{2}\log(1-r^2).
\end{equation}
Equivalently,
$$(1-r^2)^{(n+2)/2}\leq e^{-\Lambda_{B_2^n}^*(re_1)}\leq (1-r^2)^{n/2}.$$
On the other hand, an explicit computation of the half-space depth in the Euclidean ball gives
$$q_{B_2^n}(re_1)= (1-r^2)^{(n+1)/2}h(r,n),$$
where
$$\frac{1}{\sqrt{2\pi(n+2)}}\leq h(r,n)\leq\frac{1}{r\sqrt{2\pi n}}.$$ See \cite[Lemma~2.2]{Bonnet-Chasapis-Grote-Temesvari-Turchi-2019}.
For fixed $r\in(0,1)$, these estimates imply
$$q_{B_2^n}(re_1)\asymp_r \frac{1}{\sqrt{n}}e^{-\Lambda_{B_2^n}^*(re_1)}.$$
\end{proof}

The preceding proposition shows that the loss of order $\sqrt{n}$ in Theorem~\ref{thm:q-lower-bound-star} is not an artifact of the proof.
It is already present for the Euclidean ball. The next observation shows, however, that boundedness of the support is essential.

\begin{remark}\label{rem:general-logconcave-counterexample}\rm 
There is no direct analogue of Theorem~\ref{thm:q-lower-bound-star} for arbitrary log-concave probability measures with the same dimensional dependence. Consider
the symmetric exponential probability measure $\nu$ on $\mathbb{R}$ and its product measure
$$\nu_n=\nu^{\otimes n}.$$
Let $x_n=ne_1$. A direct computation gives
$$\frac{q_\nu(r)}{e^{-\Lambda_\nu^*(r)}}\sim\frac{1}{er},\qquad r\to\infty.$$
Since
$$\Lambda_{\nu_n}^*(x_n)=\Lambda_\nu^*(n)$$
and
$$q_{\nu_n}(x_n)\leq q_\nu(n),$$
we obtain
$$\frac{q_{\nu_n}(x_n)}{e^{-\Lambda_{\nu_n}^*(x_n)}}\leq\frac{C}{n}.$$
Thus, in contrast with the convex-body setting, the ratio between half-space depth and the exponential of the negative Cram\'{e}r transform
can be as small as order $1/n$.
\end{remark}

The pointwise comparison has an immediate geometric consequence. For $p>0$, define the depth region
$$T_p(K)=\left\{x\in K:q_K(x)\geq e^{-p}\right\}.$$
If
$$B_p(K)=\left\{x\in K:\Lambda_K^*(x)\leq p\right\}$$
denotes the corresponding Cram\'{e}r sublevel set, then Theorem~\ref{thm:q-lower-bound-star} gives
$$T_p(K)\subseteq B_p(K)\subseteq T_{p+\frac{1}{2}\log n+C}(K).$$
Indeed, the first inclusion follows from
$$q_K(x)\geq e^{-p}\quad\Longrightarrow\quad\Lambda_K^*(x)\leq p,$$
while the second follows directly from
$$q_K(x)\geq\frac{C}{\sqrt{n}}e^{-\Lambda_K^*(x)}.$$

Thus, up to an additive $O(\log n)$ term in the level parameter, half-space depth regions and Cram\'{e}r sublevel sets describe the same
family of subsets of $K$.

\section{The Cram\'{e}r transform inside convex bodies}\label{sec:cramer-inside-convex-bodies}

We next investigate how large the Cram\'{e}r transform can be on subsets of a convex body. Throughout this section, $K$ is a centered convex body in
$\mathbb{R}^n$ and $\mu_K$ denotes the uniform probability measure on $K$. For $\eta\in(0,1)$, define
$$M_\eta=\max_{x\in\eta K}\Lambda_K^*(x).$$

The following upper bound is a direct consequence of the convexity of $K$.

\begin{lemma}\label{lem:star-max-upper-bound}
For every $\eta\in(0,1)$,
$$M_\eta\leq n\log\left(\frac{1}{1-\eta}\right).$$
In particular, for every $x\in\operatorname{int}(K)$,
$$\Lambda_K^*(x)\leq n\log\left(\frac{1}{1-\|x\|_K}\right).$$
\end{lemma}

\begin{proof}
Fix $x\in\eta K$ and $t\in\mathbb{R}^n$. Since
$$x+(1-\eta)K\subseteq K,$$
we have
$$e^{\Lambda_K(t)}=\frac{1}{\vol_n(K)}\int_K e^{\langle t,y\rangle}\,dy\geq\frac{1}{\vol_n(K)}\int_{x+(1-\eta)K}e^{\langle t,y\rangle}\,dy.$$
Making the change of variables
$$y=x+(1-\eta)z$$
gives
$$e^{\Lambda_K(t)}\geq (1-\eta)^ne^{\langle t,x\rangle}\frac{1}{\vol_n(K)}\int_Ke^{(1-\eta)\langle t,z\rangle}\,dz.$$
Since $K$ is centered, Jensen's inequality yields
$$\frac{1}{\vol_n(K)}\int_Ke^{(1-\eta)\langle t,z\rangle}\,dz\geq\exp\left((1-\eta)\frac{1}{\vol_n(K)}\int_K\langle t,z\rangle\,dz\right)=1.$$
Consequently,
$$\Lambda_K(t)\geq\langle t,x\rangle + n\log(1-\eta).$$
It follows that
$$\langle t,x\rangle-\Lambda_K(t)\leq n\log\left(\frac{1}{1-\eta}\right).$$
Taking the supremum over $t\in\mathbb{R}^n$ gives
$$\Lambda_K^*(x)\leq n\log\left(\frac{1}{1-\eta}\right).$$
Taking the maximum over $x\in\eta K$ proves the first assertion.

For the second assertion, take $\eta=\|x\|_K$. Since $x\in\eta K$, the previous estimate gives
$$\Lambda_K^*(x)\leq n\log\left(\frac{1}{1-\|x\|_K}\right).$$
\end{proof}

The preceding estimate has the correct order near the boundary. In particular, the Cram\'{e}r transform necessarily becomes large as one
approaches $\partial K$. We next establish a lower bound in the interior which, although of a different form near the boundary, shows
that $\Lambda_K^*$ already reaches order $n$ on every fixed dilation $\eta K$.

\begin{lemma}\label{lem:star-max-lower-bound}
There exists an absolute constant $c>0$ such that, for every $\eta\in(0,1)$,
$$M_\eta\geq c\eta^2 n.$$
\end{lemma}

\begin{proof}
We divide the argument into two cases.

Suppose first that $M_\eta\geq1$. Since $M_\eta=\max_{x \in \eta K}\Lambda_K^\ast(x)$,
$$\eta K\subseteq B_{M_\eta}(K),$$
where
$$B_t(K)=\{x\in K:\Lambda_K^*(x)\leq t\}$$
denotes the $t$-sublevel set of the Cram\'{e}r transform.

By \cite[Proposition~2.7]{Giannopoulos-Tziotziou-2025}, there exists an absolute constant $c_1>0$ such that
$$B_{M_\eta}(K)\subseteq 2Z_{c_1M_\eta}^{+}(\mu_K)\subseteq 2Z_{c_1M_\eta}(\mu_K).$$
If
$$M_\eta\geq\frac{n}{c_1},$$
then immediately
$$M_\eta\geq c\eta^2n.$$
We may therefore assume that
$$M_\eta<\frac{n}{c_1}.$$
Using the standard volume estimate for the bodies $Z_t(\mu_K)$ (see \cite[Theorem~5.1.17]{BGVV-book}), we obtain
$$\eta \vol_n(K)^{1/n}=\operatorname{vol}_n(\eta K)^{1/n}\leq c_2\sqrt{\frac{M_\eta}{n}}\vol_n(K)^{1/n}$$
for an absolute constant $c_2>0$. Hence
$$M_\eta\geq c\eta^2n.$$

It remains to consider the case
$$M_\eta<1.$$
Applying Lemma~\ref{lem:self-conc-inequality-local} with $f=\Lambda_K^*$ and $x=0$, and using
$$\Lambda_K^*(0)=0,\quad \nabla\Lambda_K^*(0)=0, \quad \nabla^2\Lambda_K^*(0)=\Sigma^{-1},$$
where
$$\Sigma=\operatorname{Cov}(\mu_K),$$
we obtain
$$\Lambda_K^*(y)\geq\omega\left(\sqrt{\langle\Sigma^{-1}y,y\rangle}\right),$$
for every $y \in \operatorname{int}(K)$. We claim that there exists $y_0\in\eta K$ such that
$$\langle\Sigma^{-1}y_0,y_0\rangle\geq\eta^2n.$$
Indeed,
\begin{align*}
\frac{1}{\vol_n(K)}\int_K\langle\Sigma^{-1}x,x\rangle\,dx &=\frac{1}{\vol_n(K)}\int_K\operatorname{tr}\left(\Sigma^{-1}xx^T\right)\,dx\\
&=\operatorname{tr}\left(\Sigma^{-1}\frac{1}{\vol_n(K)}\int_Kxx^T\,dx\right)\\
&=\operatorname{tr}(\Sigma^{-1}\Sigma)=n.
\end{align*}
Thus, there exists $y_0\in\eta K$ such that
$$\langle\Sigma^{-1}y_0,y_0\rangle\geq\eta^2n.$$
Consequently,
$$M_\eta\geq\Lambda_K^*(y_0)\geq\omega(\eta\sqrt n),$$
as $\omega$ is non-decreasing. 

Since $M_\eta<1$, let $s_0>0$ be the unique number satisfying
$$\omega(s_0)=1.$$
Then
$$\eta\sqrt n\leq s_0.$$
Because 
$$\omega(s)\geq cs^2,\qquad 0\leq s\leq s_0,$$
we conclude that
$$M_\eta\geq c\eta^2n.$$
This completes the proof.
\end{proof}

The two preceding estimates give the following useful description of the scale of the Cram\'{e}r transform on homothetic copies of $K$:
$$c\eta^2n\leq M_\eta \leq n\log\left(\frac{1}{1-\eta}\right).$$
In particular, for every fixed $\eta\in(0,1)$,
$$M_\eta\asymp_\eta n.$$
Thus, the natural scale of $\Lambda_K^*$ in the bulk of a convex body is linear in the dimension.

We will also need a lower bound that captures the behavior of the Cram\'{e}r transform in the radial direction near the boundary.

\begin{lemma}\label{lem:boundary-star-lower-bound}
For every $\eta\in(0,1)$ and every $z\in\partial(\eta K)$,
$$\Lambda_K^*(z)\geq -\log(1-\eta)-\eta.$$ In particular, \[M_\eta \geq \max \{c\eta^2n, -\log(1-\eta)-\eta\}.\]
\end{lemma}

\begin{proof}
Write
$$z=\eta y,\qquad y\in\partial K.$$
We first establish a one-dimensional estimate along the segment $[0,y)$. For $t\in[0,1)$, set
$$g(t)=\Lambda_K^*(ty).$$
By the standard Hessian characterization of the domain of a self-concordant barrier, we have
$$(1-t)^2\left\langle\nabla^2\Lambda_K^*(ty)y,y\right\rangle\geq 1.$$
Indeed, if this inequality failed, the corresponding self-concordance estimate (see \cite[Theorem~4.1.5, 1.]{book-optim}) would imply that
$y\in\operatorname{int}(K)$, contradicting $y\in\partial K$.

Therefore,
$$g''(t)=\left\langle\nabla^2\Lambda_K^*(ty)y,y\right\rangle\geq\frac{1}{(1-t)^2}.$$
Since
$$g(0)=0,\qquad g'(0)=0,$$
integrating twice gives
$$g(\eta)\geq \int_0^\eta\frac{t}{1-t}\,dt=-\log(1-\eta)-\eta.$$
Since $g(\eta)=\Lambda_K^*(z)$, the result follows.
\end{proof}

The upper and lower bounds above exhibit two distinct regimes. In the bulk, $\Lambda_K^*$ is naturally of order $n$, while its growth along
rays toward the boundary is controlled by the logarithmic singularity
$$-\log(1-\eta).$$

We record one further consequence for the Cram\'{e}r sublevel sets. For $p>0$, let
$$B_p(K)=\{x\in K:\Lambda_K^*(x)\leq p\}.$$
The upper bound in Lemma~\ref{lem:star-max-upper-bound} implies the following explicit inner inclusion:
$$\left(1-e^{-p/n}\right)K\subseteq B_p(K).$$
Indeed, if
$$x\in\left(1-e^{-p/n}\right)K,$$
then
$$\|x\|_K\leq1-e^{-p/n},$$
and hence
$$\Lambda_K^*(x)\leq n\log\left(\frac{1}{1-\|x\|_K}\right)\leq p.$$

Combining this with the comparison between half-space depth and the Cram\'{e}r transform from Theorem~\ref{thm:q-lower-bound-star}, we also obtain
$$\left(1- \left(\frac{e^{-p}\sqrt{n}}{C}\right)^{1/n}\right)K\subseteq T_p(K),$$
whenever the scalar factor on the left-hand side is non-negative. Thus, the radial geometry of the depth regions can be controlled directly through the Cram\'{e}r transform.

\section{Moments of the Cram\'{e}r transform}\label{sec:moments-cramer-transform}

We next study the size of the Cram\'{e}r transform in $L^p$. The preceding section shows that, for a convex body, $\Lambda_K^*$ is of order $n$ in
the bulk, while it can grow logarithmically in the inverse distance to the boundary. Since points chosen uniformly from a high-dimensional
convex body typically lie close to its boundary, this boundary behavior is relevant for the moments of $\Lambda_K^*$.

We begin with a general upper bound which is uniform over all convex bodies.

\begin{proposition}\label{prop:convex-body-moments}
Let $K$ be a convex body in $\mathbb{R}^n$. Then, for every $p\geq 1$,
$$\|\Lambda_K^*\|_{L^p(\mu_K)}\leq Cn(\log n+p),$$
where $C>0$ is an absolute constant.
\end{proposition}

\begin{proof}
Due to affine invariance, we can assume that $K$ is centered and with volume $1$. By Lemma~\ref{lem:star-max-upper-bound},
$$\Lambda_K^*(x)\leq n\log\left(\frac{1}{1-\|x\|_K}\right),\qquad x\in K.$$
We use polar integration with respect to the Minkowski functional:
$$\int_K f(\|x\|_K)\,dx=n\vol_n(K)\int_0^1 r^{n-1}f(r)\,dr$$
for every non-negative measurable function $f$. Since $\vol_n(K)=1$
$$\|\Lambda_K^*\|_{L^p(\mu_K)}^p\leq n^{p+1}\int_0^1r^{n-1}\log^p\left(\frac{1}{1-r}\right)\,dr.$$
Making the change of variables $u=1-r$, we obtain
$$\|\Lambda_K^*\|_{L^p(\mu_K)}^p\leq n^{p+1}\int_0^1(1-u)^{n-1}\log^p\left(\frac{1}{u}\right)\,du.$$
Splitting the integral at $u=1/n$ and using $(1-u)^{n-1}\leq 1$, we get
\begin{align*}
\|\Lambda_K^*\|_{L^p(\mu_K)}^p
&\leq n^{p+1}\int_0^{1/n}\log^p\left(\frac{1}{u}\right)\,du\\
&\quad + n^{p+1}\int_{1/n}^1(1-u)^{n-1}\log^p\left(\frac{1}{u}\right)\,du\\
&\leq n^{p+1}\int_0^{1/n}\log^p\left(\frac{1}{u}\right)\,du +(n\log n)^p.
\end{align*}
In the first integral, put $r=nu$. Then
$$n^{p+1}\int_0^{1/n}\log^p\left(\frac{1}{u}\right)\,du= n^p\int_0^1\left(\log n+\log\frac{1}{r}\right)^p\,dr.$$
By Minkowski's inequality,
\begin{align*}
\left(\int_0^1\left(\log n+\log\frac{1}{r}\right)^p\,dr\right)^{1/p}
&\leq\log n +\left(\int_0^1\log^p\frac{1}{r}\,dr\right)^{1/p}\\
&=\log n+\Gamma(p+1)^{1/p}.
\end{align*}
The standard estimate
$$\Gamma(p+1)^{1/p}\leq Cp,\qquad p\geq 1,$$
therefore gives
$$\|\Lambda_K^*\|_{L^p(\mu_K)}\leq Cn(\log n+p).$$
\end{proof}

The dependence on both $n$ and $p$ in Proposition~\ref{prop:convex-body-moments} is optimal, as can already be seen for Euclidean balls.

\begin{proposition}\label{prop:ball-cramer-moments}
For every $p\geq 1$ we have that
$$\|\Lambda_{B_2^n}^*\|_{L^p(\mu_{B_2^n})}\asymp n(\log n+p),$$
with absolute implicit constants.
\end{proposition}

\begin{proof}
By rotational invariance,
$$\Lambda_{B_2^n}^*(x)=\Lambda_{X_1}^*(\|x\|_2),$$
where $X$ is uniformly distributed on $B_2^n$.

The estimates obtained from the Bessel-function representation of the Cram\'{e}r transform in \eqref{eq:unit-ball-star-bounds} 
give, uniformly for $0\leq r<1$,
$$\frac{n}{2}\log\frac{1}{1-r^2}\leq\Lambda_{B_2^n}^*(re_1)\leq\frac{n+2}{2}\log\frac{1}{1-r^2}.$$
Consequently,
$$\|\Lambda_{B_2^n}^*\|_{L^p(\mu_{B_2^n})}\asymp n\left(\int_0^1nr^{n-1}\log^p\left(\frac{1}{1-r^2}\right)\,dr\right)^{1/p}.$$
Let $R$ be a random variable with density
$$f_R(r)=nr^{n-1}\mathbf{1}_{(0,1)}(r),$$
and define
$$Y=-\log(1-R^2).$$
Then
$$\|\Lambda_{B_2^n}^*\|_{L^p(\mu_{D_n})}\asymp n\|Y\|_{L^p}.$$
Moreover, for $t>0$,
$$\mathbb{P}(Y>t)=1-(1-e^{-t})^{n/2}.$$
Using $1-(1-u)^m\leq mu$ for $u\in[0,1]$, we obtain
$$\mathbb{P}(Y>t)\leq\frac{n}{2}e^{-t}.$$
Thus
$$Y\preceq_{\mathrm{st}}\log(n/2)+E,$$
where $E$ is an exponential random variable with parameter $1$. Consequently,
$$\|Y\|_{L^p}\leq\log(n/2)+\|E\|_{L^p}\leq C(\log n+p).$$
For the reverse inequality, if
$$t\geq\log(n/2),$$
then $ne^{-t}/2\leq 1$, and the elementary estimate
$$1-(1-u)^m\geq\frac{mu}{2},\qquad 0\leq mu\leq 1,$$
gives
$$\mathbb{P}(Y>t)\geq\frac{n}{4}e^{-t}.$$
Taking
$$t=\log(n/2)+p$$
and using Markov's inequality we obtain
$$\|Y\|_{L^p}\geq t\,\mathbb{P}(Y>t)^{1/p},$$
which implies that
$$\|Y\|_{L^p}\geq\frac{1}{2e}\bigl(\log(n/2)+p\bigr).$$
Therefore,
$$\|Y\|_{L^p}\asymp\log n+p,$$
and the proposition follows.
\end{proof}

The preceding result identifies the natural scale of the moments: the contribution of the typical distance to the boundary is responsible
for the $\log n$ term, while the exponential tail of the boundary singularity produces the linear dependence on $p$.

It is also useful to record an exponential moment estimate. For convex bodies, the upper bound on $\Lambda_K^*$ gives a particularly explicit
calculation.

\begin{proposition}\label{prop:convex-body-exponential-moment}
Let $K$ be a body in $\mathbb{R}^n$. Then, for every $a\in(0,1)$,
$$\int_K\exp\left(\frac{a}{n}\Lambda_K^*(x)\right)\,dx\leq 2^a\Gamma(1-a)n^a.$$
In particular,
$$\mathbb{E}_{\mu_K}\left[\exp\left(\frac{a}{n}\Lambda_K^*(X)\right)\right]\leq C(a)n^a.$$
\end{proposition}

\begin{proof}
We can assume that $K$ is centered with volume $1$. By Lemma~\ref{lem:star-max-upper-bound},
$$\Lambda_K^*(x)\leq n\log\left(\frac{1}{1-\|x\|_K}\right).$$
Therefore, using polar integration,
\begin{align*}
\int_K\exp\left(\frac{a}{n}\Lambda_K^*(x)\right)\,dx
&\leq n\int_0^1r^{n-1}(1-r)^{-a}\,dr\\
&=nB(n,1-a).
\end{align*}
The beta-gamma identity gives
$$nB(n,1-a)=\Gamma(1-a)\frac{\Gamma(n+1)}{\Gamma(n+1-a)}.$$
Using the standard estimate
$$\frac{\Gamma(n+1)}{\Gamma(n+1-a)}\leq (n+1)^a\leq 2^a n^a,$$
we obtain
$$\int_K\exp\left(\frac{a}{n}\Lambda_K^*(x)\right)\,dx\leq 2^a\Gamma(1-a)n^a.$$
\end{proof}

For the Euclidean ball, the preceding estimate is not sharp in its dependence on $n$. Indeed, the lower estimate for the Cram\'{e}r transform
gives a non-trivial lower bound with a smaller power of $n$.

\begin{remark}\label{rem:ball-exponential-moment}\rm 
The lower bound
$$\Lambda_{B_2^n}^*(re_1)\geq\frac{n}{2}\log\left(\frac{1}{1-r^2}\right)$$
implies that, for every $a\in(0,1)$,
$$\mathbb{E}_{\mu_{B_2^n}}\left[\exp\left(\frac{a}{n}\Lambda_{B_2^n}^*(X)\right)\right]\geq C(a)n^{a/2}.$$
Thus the exponential moment of $\Lambda_K^*/n$ can grow polynomially in the dimension, even for the Euclidean ball.
\end{remark}

We next turn to general log-concave probability measures. The estimates above are based on the fact that the support is a convex body and hence
do not directly apply when the measure has unbounded support. A different argument, based on appropriate high-probability convex
superlevel sets of the density, gives the following uniform bound. 

The first part of the proof of the following proposition is essentially \cite[Lemma~3.5]{Giannopoulos-Tziotziou-2025}, which was the main ingredient in proving the finiteness of all moments of $\Lambda_\mu^\ast$.

\begin{proposition}\label{prop:log-concave-moments}
There exists an absolute constant $C>0$ such that, for every log-concave probability measure $\mu$ on $\mathbb{R}^n$ and every $u\geq 0$,
$$\mu\left(\left\{x:\Lambda_\mu^*(x)>Cn(\log n+u)\right\}\right)\leq e^{-u}.$$
Consequently, for every $p\geq 1$,
$$\|\Lambda_\mu^*\|_{L^p(\mu)}\leq C_1n(\log n+p),$$
where $C_1>0$ is an absolute constant.
\end{proposition}

\begin{proof}
After an affine transformation, we may assume that $\mu$ is isotropic. Let $$R_t(\mu)=\{x: f_\mu(x) \geq e^{-t}f_\mu(0)\},$$ where $f_\mu$ is the log-concave density of $\mu$.  By
\cite[Proposition~2.1, Lemma~3.4]{Giannopoulos-Tziotziou-2025}, for $t\geq 20n$,
$$\mu(R_t(\mu))\geq 1-e^{-t/4},\qquad cB_2^n\subseteq R_t(\mu),$$
where $c>0$ is an absolute constant.

Fix $0<\delta<1/2$. If
$$x\in(1-\delta)R_t(\mu),$$
then convexity gives
$$B(x,c\delta)\subseteq R_t(\mu).$$
Every half-space passing through $x$ therefore contains at least half of this ball. Thus, 
$$\mu(H) \geq \int_{H \cap B(x,c\delta)}f_\mu(x) \, dx \geq e^{-t}f_\mu(0)c^n\delta^n\vol_n(B_2^n).$$ 
Taking infimum over all those half-spaces yields 
$$q_\mu(x) \geq e^{-t}f_\mu(0)c^n\delta^n\vol_n(B_2^n).$$ 
For an isotropic log-concave probability measure,
$$f_\mu(0)\geq c_1^n$$
for an absolute constant $c_1>0$. Since
$$\vol_n(B_2^n)\geq\left(\frac{c_2}{\sqrt n}\right)^n,$$
we obtain
$$q_\mu(x)\geq e^{-t}\left(\frac{c_3\delta}{\sqrt n}\right)^n.$$
By the general comparison
$$\Lambda_\mu^*(x)\leq-\log q_\mu(x),$$
it follows that
$$\Lambda_\mu^*(x)\leq t+\frac{n}{2}\log(c_4n)+n\log\frac{1}{\delta}.$$
Thus, for
$$t=20n+4u,\qquad \delta=\frac{e^{-u}}{8n},$$
there exists an absolute constant $C>0$ such that
$$\Lambda_\mu^*(x)\leq Cn(\log n+u)$$
throughout $(1-\delta)R_t(\mu)$.

It remains to estimate the measure of this set. Put
$$\kappa=\frac{\delta}{1-\delta}.$$
Then
$$(1+\kappa)(1-\delta)R_t(\mu)=R_t(\mu).$$
By the standard homothetic estimate for log-concave measures (see \cite[Lemma~3.1]{Giannopoulos-Tziotziou-2025}),
$$\mu(R_t(\mu))\leq e^{2n\kappa}\mu((1-\delta)R_t(\mu)).$$
Since $\kappa\leq 2\delta$, we obtain
$$\mu((1-\delta)R_t(\mu))\geq e^{-4n\delta}\bigl(1-e^{-t/4}\bigr).$$
For our choice of $t$ and $\delta$,
$$4n\delta=\frac{e^{-u}}{2}$$
and
$$e^{-t/4}=e^{-5n-u}.$$
Therefore,
$$\mu((1-\delta)R_t(\mu))\geq e^{-e^{-u}/2}\left(1-e^{-5n-u}\right)\geq 1-e^{-u}.$$
Hence
$$\mu\left(\left\{x:\Lambda_\mu^*(x)>Cn(\log n+u)\right\}\right)\leq e^{-u}.$$

The last estimate implies that
$$\Lambda_\mu^*(X)\preceq_{\mathrm{st}} Cn(\log n+E)$$
where $E$ is an exponential random variable with parameter $1$. Hence
$$\|\Lambda_\mu^*\|_{L^p(\mu)}\leq Cn\bigl(\log n+\|E\|_{L^p}\bigr)\leq C_1n(\log n+p).$$
\end{proof}

The preceding proposition shows that the moment estimate
$$\|\Lambda_\mu^*\|_{L^p(\mu)}\lesssim n(\log n+p)$$
extends from uniform measures on convex bodies to arbitrary log-concave probability measures. We can obtain a corresponding
exponential integrability statement.

\begin{proposition}\label{prop:log-concave-exponential-moment}
Let $\mu$ be a log-concave probability measure on $\mathbb{R}^n$. Then, for every $a\in(0,1/5)$,
$$\mathbb{E}_\mu\left[\exp\left(\frac{a}{n}\Lambda_\mu^*\right)\right]\leq C(a)n^{3a/2},$$
where $C(a)>0$ depends only on $a$.

More precisely,
$$\|\Lambda_\mu^*\|_{L^{\psi_1}(\mu)}\leq Cn\log n.$$
\end{proposition}

\begin{proof}
We use the more precise form of the preceding argument. With
$$t=20n+4u,\qquad \delta=\frac{e^{-u}}{8n},$$
the estimate above shows that
$$\Lambda_\mu^*(x)\leq A_n+(n+4)u$$
on $(1-\delta)R_t(\mu)$, where
$$A_n=20n+\frac{n}{2}\log(c_2n)+n\log(8n).$$
Consequently,
$$\mu\left(\left\{x:\Lambda_\mu^*(x)>A_n+s\right\}\right)\leq\exp\left(-\frac{s}{n+4}\right),\qquad s>0.$$
For $\lambda>0$, integration gives
$$\mathbb{E}_\mu\left[e^{\lambda\Lambda_\mu^*}\right]\leq e^{\lambda A_n} +\lambda\int_{A_n}^{\infty}e^{\lambda t}\mu(\{x:\Lambda_\mu^*(x)>t\})\,dt.$$
Writing $t=A_n+s$ and using the tail estimate,
\begin{align*}
\mathbb{E}_\mu\left[e^{\lambda\Lambda_\mu^*}\right]
&\leq e^{\lambda A_n}+\lambda e^{\lambda A_n}\int_0^\infty e^{(\lambda-1/(n+4))s}\,ds\\
&=\frac{e^{\lambda A_n}}{1-\lambda(n+4)},
\end{align*}
provided that
$$\lambda<\frac{1}{n+4}.$$
Taking
$$\lambda=\frac{a}{n},\qquad 0<a<\frac15,$$
we obtain
$$\mathbb{E}_\mu\left[\exp\left(\frac{a}{n}\Lambda_\mu^*\right)\right]\leq\frac{\exp(aA_n/n)}{1-5a}\leq C(a)n^{3a/2},$$
where $C(a)=c^a/(1-5a)$ for some absolute constant $c>0$. The $\psi_1$ estimate follows if we set
$a=c_1/\log n$ for a suitable absolute constant $c_1>0$.
\end{proof}

\begin{remark}\rm  The exponential moment estimates for convex bodies and general log-concave measures are of somewhat different strength. For convex
bodies, the direct boundary estimate gives
$$\mathbb{E}\exp\left(\frac{a}{n}\Lambda_K^*\right)\lesssim_a n^a,\qquad 0<a<1,$$
whereas the argument available for general log-concave measures gives
$$\mathbb{E}\exp\left(\frac{a}{n}\Lambda_\mu^*\right)\lesssim_a n^{3a/2},\qquad 0<a<\frac15.$$
The latter estimate is sufficient for the moment bounds above, but it leaves open the possibility that the exponential integrability for
general log-concave measures can be improved.
\end{remark}

\section{\texorpdfstring{$s$}{s}-Concavity and half-space depth}\label{sec:s-concavity-half-space-depth}

We now turn to the distribution of the half-space depth. The key observation is that, for $s$-concave measures, a suitable power of the
depth is itself a concave function. This allows us to apply a classical integral inequality for concave functions and obtain a
dimension-free concentration estimate for the logarithm of the depth.

Recall that a probability measure $\mu$ on $\mathbb{R}^n$ is $s$-concave, for $s \in(0,1/n]$, if
$$\mu\bigl((1-\lambda)A+\lambda B\bigr)^s\geq (1-\lambda)\mu(A)^s+\lambda\mu(B)^s$$
for all non-empty compact sets $A,B\subseteq\mathbb{R}^n$ and every $\lambda\in[0,1]$. Throughout this section, $\mu$ is supported on a convex body $K$.

For $u\in S^{n-1}$ and $x\in K$, define
$$H_u(x)=\bigl\{y\in\mathbb{R}^n:\langle y,u\rangle\geq\langle x,u\rangle\bigr\}.$$
Thus, $H_u(x)$ is the closed half-space whose boundary passes through $x$ and whose outer normal is $u$. Recall that the Tukey half-space depth is
$$q_\mu(x)=\inf_{u\in S^{n-1}}\mu\bigl(H_u(x)\bigr).$$

We begin with a direct consequence of \eqref{eq:dimension-free} and Proposition~\ref{prop:log-concave-moments}:

\begin{lemma}\label{lem:tukey-finite-moments}
    There exists an absolute constant $C>0$, such that for every log-concave probability measure $\mu$ on $\mathbb{R}^n$ and every $p \geq 1$, \[\|-\log q_\mu\|_{L^p(\mu)} \leq Cn(\log n+p).\] In particular, $-\log q_\mu$ has finite moments of all orders.
\end{lemma}

The following elementary observation is the main reason that $s$-concavity is particularly well suited to the study of half-space depth.

\begin{lemma}\label{lem:depth-s-concave}
Let $\mu$ be an $s$-concave probability measure supported on a convex body $K$. Then $q_\mu^s$ is concave on $K$.
\end{lemma}

\begin{proof}
Fix $x,y\in K$ and $\lambda\in[0,1]$, and set
$$z=(1-\lambda)x+\lambda y.$$
For every $u\in S^{n-1}$, the convexity of half-spaces gives
$$(1-\lambda)H_u(x)+\lambda H_u(y)\subseteq H_u(z).$$
Hence, by the monotonicity and the $s$-concavity of $\mu$,
\begin{align*}
\mu\bigl(H_u(z)\bigr)^s &\geq\mu\bigl((1-\lambda)H_u(x)+\lambda H_u(y)\bigr)^s\\
&\geq (1-\lambda)\mu\bigl(H_u(x)\bigr)^s +\lambda\mu\bigl(H_u(y)\bigr)^s\\
&\geq (1-\lambda)q_\mu(x)^s +\lambda q_\mu(y)^s.
\end{align*}
Taking the infimum over $u\in S^{n-1}$ yields
$$q_\mu(z)^s\geq (1-\lambda)q_\mu(x)^s +\lambda q_\mu(y)^s.$$
Therefore, $q_\mu^s$ is concave on $K$.
\end{proof}

This concavity has a useful consequence for the fluctuations of the depth. We shall use a more general version of Berwald's inequality that appears in \cite[Theorem~6.2]{Fra-Li-Mad}.

\begin{theorem}\label{thm:Berwald-inequality}
Let $\mu$ be an $s$-concave probability measure supported on a convex body  $K$. If $g:K\to[0,\infty)$ is concave, then the function
$$p\longmapsto\left(\frac{1}{pB(p,1+s^{-1})}\int_K g(x)^p\,d\mu(x)\right)^{1/p}$$
is non-increasing on $(-1,\infty)$.
\end{theorem}

We apply the extended version of Berwald's inequality to the concave function $q_\mu^s$. The resulting estimate is particularly striking: although the individual depth can vary 
exponentially in the dimension, its logarithm has fluctuations of order $1/s$.

\begin{theorem}\label{thm:variance-log-depth}
Let $\mu$ be an $s$-concave probability measure supported on a convex body $K$. Then
$$\operatorname{Var}_\mu\bigl(-\log q_\mu\bigr)\leq\frac{C}{s^2},$$
where $C>0$ is an absolute constant.

In particular, if $\mu=\mu_K$ is the uniform probability measure on a convex body $K$, then
$$\operatorname{Var}_{\mu_K}\bigl(-\log q_K\bigr)\leq Cn^2.$$
\end{theorem}

\begin{proof}
Set
$$g(x)=q_\mu(x)^s.$$
By Lemma~\ref{lem:depth-s-concave}, $g$ is concave.
Set $a=1+s^{-1}>1$. Since for every $p>0$ \[pB(p,a)=\frac{\Gamma(p+1)\Gamma(a)}{\Gamma(p+a)},\] and the function on the right-hand side is also defined on $0$, we set
$$A(p)=-\log \left(\frac{\Gamma(p+1)}{\Gamma(p+a)}\right)-\log\Gamma(a), \qquad C(p)=A(p) +\log\int_K g(x)^p\,d\mu(x),$$ for every $p \geq 0$. Observe that \[A'(p)=-\psi(p+1)+\psi(p+a)\] for every $p \geq 0$, where $\psi$ is the digamma function. In particular (see \cite{Artin-book}),  \[A''(0)=\psi'(a)-\frac{\pi^2}{6}.\] Moreover by Lemma~\ref{lem:tukey-finite-moments}, $C(p)$ has derivatives of every order at $0$ and more precisely, \[C''(0)=\psi'(a)-\frac{\pi^2}{6}+\operatorname{Var}_\mu(\log g).\]
Set now for $p>0$ $$H(p)=\frac{C(p)}{p}.$$
Notice that $H$ has a finite limit as $p\to0^+$. So it extends continuously to $[0,\infty)$ and by Theorem~\ref{thm:Berwald-inequality}, $H$ is non-increasing on $[0,+\infty)$. Moreover, $H'(0)$ exists and is equal to $C''(0)/2$. Thus,
$$C''(0)\leq0.$$ 
It follows that
$$\operatorname{Var}_\mu(\log g)\leq \frac{\pi^2}{6}-\psi'(a) \leq \frac{\pi^2}{6},$$ since $\psi'>0$. Finally,
$$\log g=s\log q_\mu,$$
and therefore
$$s^2\operatorname{Var}_\mu(\log q_\mu)\leq\frac{\pi^2}{6}.$$
Since variance is unchanged by multiplication by $-1$,
$$\operatorname{Var}_\mu(-\log q_\mu)\leq\frac{\pi^2}{6s^2}.$$
This proves the first assertion.

For the uniform probability measure on a convex body, the corresponding concavity parameter is $s=1/n$. Hence
$$\operatorname{Var}_{\mu_K}(-\log q_K)\leq\frac{\pi^2}{6}n^2.$$
\end{proof}

\begin{remark}\rm The upper bound for the variance of $I_K(x)=-\log q_K(x)$ is optimal with respect to the dimension. Indeed, if $K=B_2^n$ and $n$ is even, using the estimates for $q_{B_2^n}$ in Proposition~\ref{prop:optimality-of-inequality}, a direct calculation shows that 
$$\mathbb{E}_{\mu_{B_2^n}}(I_{B_2^n})=\frac{n+1}{2}H_{\frac{n}{2}}+O(\log n),$$ 
where $H_m=\sum_{k=1}^m \frac{1}{k}$ is the $m$-th harmonic number. In addition, 
$$\mathbb{E}_{\mu_{B_2^n}}(I_{B_2^n}^2)\geq\frac{(n+1)^2}{4}H_{\frac{n}{2}}^2+\frac{(n+1)^2}{4}+O(n \log^2 n).$$ 
Hence, we obtain the lower bound 
$$\Var_{\mu_{B_2^n}}(I_{B_2^n}) \geq \frac{n^2}{4}+O(n \log^2n).$$
\end{remark}

\section{Threshold phenomena for random convex hulls}\label{sec:threshold-random-convex-hulls}

We now apply the preceding estimates to the random convex hull problem. Let $X_1,X_2,\ldots$ be independent random vectors distributed according
to a probability measure $\mu$ on $\mathbb{R}^n$, and write
$$K_N=\operatorname{conv}\{X_1,\ldots,X_N\}.$$
For $x\in\mathbb{R}^n$, define
$$N_\mu(x)=\inf\left\{N\in\mathbb{N}:\mathbb{P}(x\in K_N)\geq\frac12\right\}.$$
Thus, $N_\mu(x)$ measures the number of random points required, up to a fixed probability threshold, for their convex hull to contain $x$.

A fundamental result of Hayakawa, Lyons, and Oberhauser relates this quantity directly to the Tukey half-space depth.

\begin{theorem}[Hayakawa--Lyons--Oberhauser]\label{thm:HLO-threshold}
Let $\mu$ be a probability measure on $\mathbb{R}^n$. Then, for every $x\in\mathbb{R}^n$,
$$\frac{1}{2}\leq N_\mu(x)q_\mu(x)\leq 3n+1.$$
\end{theorem}

See \cite[Theorem~16]{HLO}. In particular, the number of points required to capture $x$ is determined, up to a polynomial factor in
the dimension, by the reciprocal of its half-space depth.

For uniform measures on convex bodies, the results of the previous sections show that the half-space depth is closely related to the
Cram\'{e}r transform. More precisely, Theorem~\ref{thm:q-lower-bound-star} gives
$$q_K(x)\geq\frac{C}{\sqrt{n}}e^{-\Lambda_K^*(x)},$$
while the general comparison
$$q_K(x)\leq e^{-\Lambda_K^*(x)}$$
holds for every $x\in K$. Consequently,
$$e^{\Lambda_K^*(x)}\leq\frac{1}{q_K(x)}\leq\frac{\sqrt{n}}{C}e^{\Lambda_K^*(x)}.$$
Combining this with Theorem~\ref{thm:HLO-threshold} yields
$$\frac12 e^{\Lambda_K^*(x)}\leq N_K(x)\leq Cn^{3/2}e^{\Lambda_K^*(x)}.$$
Thus, up to polynomial factors in $n$, the Cram\'{e}r transform is the logarithmic scale governing the number of random points required to
capture a given point.

We next quantify this statement uniformly over homothetic copies of a centered convex body $K$. Recall that
$$M_\eta=\max_{x\in\eta K}\Lambda_K^*(x),\qquad 0<\eta<1.$$
The estimates established in Section~\ref{sec:cramer-inside-convex-bodies} show that
$$c\eta^2n\leq M_\eta\leq n\log\left(\frac{1}{1-\eta}\right).$$
This leads immediately to the following result.

\begin{proposition}\label{prop:threshold-number-points}
There exist absolute constants $c,C>0$ such that, for every centered convex body $K$ in $\mathbb{R}^n$ and every $\eta\in(0,1)$,
$$\frac12e^{c\eta^2n}\leq\sup_{x\in\eta K}N_K(x)\leq Cn^{3/2}(1-\eta)^{-n}.$$
\end{proposition}

\begin{proof}
Let $x\in\eta K$. By Theorem~\ref{thm:HLO-threshold},
$$N_K(x)\leq\frac{3n+1}{q_K(x)}.$$
Using the lower bound for the half-space depth,
$$q_K(x)\geq\frac{C}{\sqrt{n}}e^{-\Lambda_K^*(x)},$$
we obtain
$$N_K(x)\leq Cn^{3/2}e^{\Lambda_K^*(x)}\leq Cn^{3/2}e^{M_\eta}.$$
The upper bound for $M_\eta$ gives
$$N_K(x)\leq Cn^{3/2}\left(\frac{1}{1-\eta}\right)^n.$$
Taking the supremum over $x\in\eta K$ proves the upper bound.

For the converse, Lemma~\ref{lem:star-max-lower-bound} provides a point $x_\eta\in\eta K$ such that
$$\Lambda_K^*(x_\eta)\geq c\eta^2n.$$
Theorem~\ref{thm:HLO-threshold} gives
$$N_K(x_\eta)\geq\frac{1}{2q_K(x_\eta)}.$$
Since
$$q_K(x_\eta)\leq e^{-\Lambda_K^*(x_\eta)},$$
we conclude that
$$N_K(x_\eta)\geq\frac12 e^{\Lambda_K^*(x_\eta)}\geq\frac12e^{c\eta^2n}.$$
Therefore,
$$\sup_{x\in\eta K}N_K(x)\geq\frac12e^{c\eta^2n}.$$
\end{proof}

The proposition exhibits two different scales. For a fixed $\eta\in(0,1)$, the lower bound is exponential in $n$, showing that there are always points in $\eta K$ which 
are exponentially difficult to capture by a random convex hull. The upper bound has the expected boundary singularity $(1-\eta)^{-n}$ and is uniform over the entire
homothetic copy $\eta K$.

The connection with threshold phenomena becomes particularly transparent when one considers the fluctuations of the Cram\'{e}r
transform and of the information content associated with the half-space depth. Define
$$I_K(x)=-\log q_K(x)$$
and
$$\beta(\mu_K)=\frac{\operatorname{Var}_{\mu_K}(\Lambda_K^*)}{\bigl(\mathbb{E}_{\mu_K}[\Lambda_K^*]\bigr)^2},\qquad \tau(\mu_K)=
\frac{\operatorname{Var}_{\mu}(I_K)}{\bigl(\mathbb{E}_{\mu_K}[I_K]\bigr)^2}.$$

In order to compare these two parameters we will first need to gather some estimates that first appeared in \cite{BGP-threshold} and \cite{Giannopoulos-Tziotziou-2025}. We take into account the solution of the slicing problem \cite{KL}, i.e. for every log-concave probability measure $\mu$ on $\mathbb{R}^n$ \[L_\mu \asymp 1.\]

\begin{lemma}\label{lem:mean-variance-star-estimates}
 If $\mu$ is a log-concave probability measure on $\mathbb{R}^n$ then \[c_1n \leq \mathbb{E}_\mu[\Lambda_\mu^\ast] \leq c_2n\log n.\] In addition, if $K$ is a convex body in $\mathbb{R}^n$, we have that\[\Var_{\mu_K}(\Lambda_K^\ast) \geq \frac{c_3}{n^2} \bigl(\mathbb{E}_{\mu_K}[\Lambda_K^\ast]\bigr)^2.\] Therefore, \[\beta(\mu_K) \geq \frac{c_3}{n^2},\] where $c_1,c_2,c_3>0$ are absolute constants.
\end{lemma}

\begin{proof}
    For the first assertion, the upper bound follows from Proposition~\ref{prop:log-concave-moments}. For the lower bound, by \cite[Lemma~5.1]{BGP-threshold}, that was essentially proven in \cite{BGP-depth}, there exist some absolute constants $n_0 \in \mathbb{N}$ and $c>0$ such that for every $n \geq n_0$, \[\mathbb{E}_{\mu}[\Lambda_\mu^\ast] \geq cn.\] Now, if $n \leq n_0$, let $u \in S^{n-1}$ and consider the log-concave random variable \[Y=\langle X,u \rangle,\] where $X$ is distributed according to $\mu$. Observe that \[\Lambda_\mu^\ast(x) \geq \Lambda_Y^\ast(\langle x,u\rangle)\] for every $x \in \mathbb{R}^n$. Hence,  \[\mathbb{E}_\mu[\Lambda_\mu^\ast]\geq \mathbb{E}_Y[\Lambda_Y^\ast].\] The one-dimensional lower bound from \cite[Lemma~5.5]{Brazitikos-Chasapis-2024} completes the proof.

    The second assertion, comes from the proof of \cite[Lemma~5.2]{BGP-threshold}.
\end{proof}

 The next proposition shows that $\beta(\mu_K)$ and $\tau(\mu_K)$ are essentially equivalent. In particular, both are upper bounded by absolute constants.

\begin{proposition}\label{prop:equivalence-q-star}
Let $K$ be a convex body. Then
$$\operatorname{Var}_{\mu_K}(\Lambda_K^*)\leq Cn^2,$$
where $C>0$ is an absolute constant. Moreover,
$$\beta(\mu_K)=\left(\tau(\mu_K)+O\left(\frac{\log n}{n}\right)\right)\left(1+O\left(\frac{\log n}{n}\right)\right).$$
In particular,
$$\beta(\mu_K)\leq C_1,\qquad \tau(\mu_K)\leq C_2,$$
where $C_1,C_2>0$ are absolute constants.
\end{proposition}

\begin{proof}
By Theorem~\ref{thm:q-lower-bound-star}, for every $x\in\operatorname{int}(K)$,
$$I_K(x)-C\log n\leq\Lambda_K^*(x)\leq I_K(x).$$
Thus, writing
$$\Lambda_K^*=I_K+h,\qquad \|h\|_{L^\infty(K)}\leq C\log n,$$
we have
$$\operatorname{Var}_{\mu_K}(\Lambda_K^*)=\operatorname{Var}_{\mu_K}(I_K)+O\left(\mathbb{E}_{\mu_K}[\Lambda_K^*]\log n+\log^2 n\right).$$
By Theorem~\ref{thm:variance-log-depth},
$$\operatorname{Var}_{\mu_K}(I_K)\leq Cn^2.$$
Furthermore, the first-moment estimate
from Lemma~\ref{lem:mean-variance-star-estimates} implies that
$$\operatorname{Var}_{\mu_K}(\Lambda_K^*)\leq Cn^2.$$
This proves the first assertion and, in particular, the upper bound for $\beta(\mu_K)$.

Again by Lemma~\ref{lem:mean-variance-star-estimates} and
$$0\leq I_K-\Lambda_K^*\leq C\log n,$$
we obtain
$$\frac{\mathbb{E}_{\mu_K}[I_K]}{\mathbb{E}_{\mu_K}[\Lambda_K^*]}=1+O\left(\frac{\log n}{n}\right).$$
Consequently,
\begin{align*}
\beta(\mu_K) &=\frac{\operatorname{Var}_{\mu_K}(\Lambda_K^*)}{\bigl(\mathbb{E}_{\mu_K}[\Lambda_K^*]\bigr)^2}\\
&=\frac{\operatorname{Var}_{\mu_K}(I_K)+O\left(\mathbb{E}_{\mu_K}[\Lambda_K^*]\log n+\log^2 n\right)}{\bigl(\mathbb{E}_{\mu_K}[\Lambda_K^*]\bigr)^2}\\
&=\left(\frac{\operatorname{Var}_{\mu_K}(I_K)}{\bigl(\mathbb{E}_{\mu_K}[I_K]\bigr)^2}+O\left(\frac{\log n}{n}\right)\right)
\left(\frac{\mathbb{E}_{\mu_K}[I_K]}{\mathbb{E}_{\mu_K}[\Lambda_K^*]}\right)^2\\
&=\left(\tau(\mu_K)+O\left(\frac{\log n}{n}\right)\right)\left(1+O\left(\frac{\log n}{n}\right)\right).
\end{align*}
This proves the claimed equivalence.

Finally, since
$$\operatorname{Var}_{\mu_K}(\Lambda_K^*)\leq Cn^2$$
and
$$\mathbb{E}_{\mu_K}[\Lambda_K^*]\geq c_1n,$$
we obtain
$$\beta(\mu_K)\leq C_1.$$
The upper bound for $\tau(\mu_K)$ follows in the same way from
$$\operatorname{Var}_{\mu_K}(I_K)\leq Cn^2$$
and
$$\mathbb{E}_{\mu_K}[I_K]\geq\mathbb{E}_{\mu_K}[\Lambda_K^*]\geq c_1n.$$
\end{proof}

Before proceeding, we need to introduce the notion of a threshold for random convex hulls. Informally, it can be understood as follows: we have a threshold if, for every fixed $\delta\in(0,1/2)$, the numbers of samples needed for the convex hulls to reach probabilities $\delta$ and $1-\delta$ are asymptotically equivalent. More precisely, given a Borel probability measure $\mu$ on $\mathbb{R}^n$ we set 
$$\varrho_1(\mu,\delta):= \sup\{r>0: \E_{\mu^N}[\mu(K_N)]\leq \delta, \hbox{ for every } N\leq \exp(r)\},$$
and
$$\varrho_2(\mu,\delta):= \inf\{r>0: \E_{\mu^N}[\mu(K_N)]\geq 1-\delta, \hbox{ for every } N\geq \exp(r)\},$$
where $\mu^N=\mu \otimes\cdots\otimes\mu \, $ ($N$-times). Let also \[\varrho(\mu_n,\delta):=\varrho_2(\mu_n,\delta)-\varrho_1(\mu_n,\delta)\] be the threshold window. With this notation, we give the following definition.

\begin{definition}\label{def:sharp-threshold}\rm 
Let $(\mu_n)_{n\in\mathbb{N}}$ be a sequence of probability measures $\mu_n$ on $\mathbb{R}^n$, and $(T_n)_{n \in \mathbb{N}}$ a sequence of positive numbers. 
We say that $(\mu_n)_{n\in\mathbb{N}}$ exhibits a sharp threshold around $(T_n)_{n \in \mathbb{N}}$ if for every $\varepsilon \in (0,1)$, there exists some $n_0=n_0(\varepsilon) \in \mathbb{N}$ such that for every $n \geq n_0$
$$\varrho_1(\mu_n,\delta_n)\geq(1-\varepsilon)T_n \qquad\hbox{ and }\qquad \varrho_2(\mu_n,\delta_n)\leq(1+\varepsilon)T_n,$$
for some sequence $(\delta_n)_{n \in \mathbb{N}}$ of positive numbers with $\lim_{n\to\infty}\delta_n=0$. Equivalently, if \[\frac{\varrho(\mu_n,\delta_n)}{T_n} \to 0,\] as $n \to \infty$.
\end{definition} We now state the threshold criterion that will be used below. It was shown in \cite[Theorem~5.5, Theorem~5.9]{BGP-threshold}.

\begin{theorem}\label{thm:rho-threshold}
Let $K$ be a convex body in $\mathbb{R}^n$ and $\delta>0$. Assume that
$$8\beta(\mu_K)<\delta<1.$$
If
$$n\geq c\log\left(\frac{2}{\delta}\right)\sqrt{\frac{\delta}{\beta(\mu_K)}},$$
then
$$\varrho(\mu_K,\delta)\leq C\sqrt{\frac{\beta(\mu_K)}{\delta}}\,\mathbb{E}_{\mu_K}[\Lambda_K^*],$$
where $c,C>0$ are absolute constants.
\end{theorem}

For uniform measures on convex bodies, Proposition~\ref{prop:equivalence-q-star} gives
$$\operatorname{Var}_{\mu_K}(\Lambda_K^*)\leq Cn^2.$$
Thus the fluctuations of $\Lambda_K^*$ are at most of order $n$. If, on the other hand, its mean is much larger than $n$, then the
Cram\'{e}r transform is concentrated on a logarithmic scale far away from the origin.

This leads to a general criterion for a sharp threshold.

\begin{theorem}\label{thm:sharp-threshold}
Let $(K_n)_{n\in\mathbb{N}}$ be a sequence of convex bodies $K_n\subseteq\mathbb{R}^n$, and let $\mu_{K_n}$ denote the corresponding
uniform probability measures. Assume that
$$\lim_{n\to\infty}\frac{1}{n}\mathbb{E}_{\mu_{K_n}}\left[\Lambda_{K_n}^*\right]=+\infty.$$
Then the sequence $(\mu_{K_n})_{n\in\mathbb{N}}$ exhibits a sharp threshold around $\bigl( \mathbb{E}_{\mu_{K_n}} [\Lambda^\ast_{K_n}]\bigr)_{n \in \mathbb{N}}$.
\end{theorem}

\begin{proof}
Set
$$m_n=\mathbb{E}_{\mu_{K_n}}[\Lambda_{K_n}^*]$$
and choose
$$\delta_n=\frac{n}{m_n}.$$
By assumption,
$$\delta_n\longrightarrow0.$$
Observe that by Proposition~\ref{prop:equivalence-q-star} \[\frac{\beta(\mu_{K_n})}{\delta_n}=\frac{\Var_{\mu_{K_n}}(\Lambda^\ast_{\mu_{K_n}})}{m_n^2\,\delta_n} \leq C\frac{n^2}{m_n^2\,\delta_n}=C\delta_n.\] Moreover, by Lemma~\ref{lem:mean-variance-star-estimates}\[\beta(\mu_K) \geq \frac{c_1}{n^2}\]Therefore, if $c>0$ is the absolute constant of Theorem~\ref{thm:rho-threshold}, then \[c\log \left(\frac{2}{\delta_n}\right)\sqrt{\frac{\delta_n}{\beta(\mu_{K_n)}}} \leq c_2 \log\left(\frac{2}{\delta_n}\right)\sqrt{\delta_n}\,n \leq n, \] for all large $n$, since $\delta_n \to 0$, as $n \to +\infty$. 
Thus, using Theorem~\ref{thm:rho-threshold} \[\frac{\rho(\mu_n,\delta_n)}{m_n}\leq C_1 \sqrt{\delta_n} \to 0,\] as $n \to \infty$.
\end{proof}

Because of Proposition~\ref{prop:equivalence-q-star}, the preceding result can equivalently be expressed in terms of Tukey's half-space depth.

\begin{corollary}\label{cor:sharp-threshold-tukey}
Let $(K_n)_{n\in\mathbb{N}}$ be a sequence of convex bodies $K_n\subseteq\mathbb{R}^n$, and let $\mu_{K_n}$ denote the corresponding
uniform probability measures. Assume that
$$\lim_{n\to\infty}\frac{1}{n}\mathbb{E}_{\mu_{K_n}}\left[I_{K_n}\right]=+\infty.$$
Then the sequence $(\mu_{K_n})_{n\in\mathbb{N}}$ exhibits a sharp threshold around $\bigl( \mathbb{E}_{\mu_{K_n}} [\Lambda^\ast_{K_n}]\bigr)_{n \in \mathbb{N}}$.
\end{corollary}

We now verify the required growth of the mean for the $\ell_p$ balls.

\begin{proposition}\label{prop:p-ball-depth-mean}
Let $p>1$. Then
$$\mathbb{E}_{\mu_{B_p^n}}[I_{B_p^n}]\geq C_p n\log n-c_p n, $$
where $C_p,c_p>0$ depend only on $p$.
\end{proposition}

\begin{proof}
Let $z\in\partial B_p^n$ and write
$$x=rz,\qquad r\in[0,1].$$
Set
$$h=1-r.$$
Consider the linear functional
$$u_z(y)=\sum_{i=1}^n\operatorname{sgn}(z_i)|z_i|^{p-1}y_i.$$
If $q$ is the conjugate exponent of $p$, then
$$\|u_z\|_q=1,\qquad u_z(z)=1.$$
Define the cap
$$C(z,h)=\left\{y\in B_p^n:u_z(y)\geq1-h\right\}.$$
Since
$$u_z(x)=r=1-h,$$
we have $x\in C(z,h)$ and therefore
$$q_{B_p^n}(rz)\leq\frac{\operatorname{vol}_n(C(z,h))}{\operatorname{vol}_n(B_p^n)}.$$

Let $y\in C(z,h)$. Then
\begin{equation}\label{eq:definition-cap-inequality}
u_z(y)\geq1-h.
\end{equation}

We distinguish two cases.

\medskip

\noindent\textit{Case 1: $1<p\leq2$.}

Set
$$t=\|y\|_p, \qquad \widetilde y=\frac{y}{t}.$$
Since
$$1-h\leq u_z(y)\leq\|y\|_p=t\leq1,$$
we have
$$1-t\leq h.$$
Moreover,
$$u_z(\widetilde y)=\frac{u_z(y)}{t}\geq u_z(y)\geq1-h.$$
Consequently,
\begin{equation}\label{eq:mean-p-norm}
\left\|\frac{z+\widetilde y}{2}\right\|_p\geq\frac{u_z(z+\widetilde y)}{2}\geq 1-\frac h2.
\end{equation}

Applying the uniform convexity inequality from \cite[Proposition~3]{Ball-Carlen-Lieb} to $(z+\tilde{y})/2$ and $(z-\tilde{y})/2$ gives
$$1\geq\left\|\frac{z+\widetilde y}{2}\right\|_p^2+\frac{p-1}{4}\|z-\widetilde y\|_p^2.$$
Together with \eqref{eq:mean-p-norm}, this yields
$$\|z-\widetilde y\|_p\leq\frac{2}{\sqrt{p-1}}\sqrt h.$$
Hence
\begin{align*}
\|y-z\|_p &\leq\|y-\widetilde y\|_p+\|\widetilde y-z\|_p\\
&=1-t+\|\widetilde y-z\|_p\\
&\leq\left(1+\frac{2}{\sqrt{p-1}}\right)\sqrt h,
\end{align*}
where we used $1-t\leq h\leq\sqrt h$. Thus,
$$C(z,h)\subseteq z+D_p\sqrt h\,B_p^n$$
for a constant $D_p>0$ depending only on $p$. Consequently,
$$q_{B_p^n}(rz)\leq D_p^n(1-r)^{n/2}.$$

\medskip
\noindent\textit{Case 2: $p>2$.}

Clarkson's inequality applied to $y$ and $z$ gives
$$\left\|\frac{y-z}{2}\right\|_p^p\leq 1-\left(\frac{u_z(y)+u_z(z)}{2}\right)^p.$$
Using \eqref{eq:definition-cap-inequality} and $u_z(z)=1$, we obtain
$$\left\|\frac{y-z}{2}\right\|_p^p\leq 1-\left(1-\frac h2\right)^p\leq\frac{ph}{2}.$$
Therefore,
$$\|y-z\|_p\leq C_p h^{1/p},$$
and hence
$$C(z,h)\subseteq z+C_p h^{1/p}B_p^n.$$
It follows that
$$q_{B_p^n}(rz)\leq C_p^n(1-r)^{n/p}.$$

In both cases, we have an estimate of the form
$$q_{B_p^n}(rz)\leq A_p^n(1-r)^{\alpha_p n},$$
where
$$\alpha_p=
\begin{cases}
1/2,&1<p\leq2,\\[2mm]
1/p,&p>2.
\end{cases}
$$
Therefore,
$$I_{B_p^n}(rz)= -\log q_{B_p^n}(rz)\geq\alpha_p n\log\frac1{1-r} -n\log A_p.$$
Using polar integration with respect to the Minkowski functional \[\|\cdot\|_{B_p^n}=\|\cdot\|_p,\] we obtain
\begin{align*}
\mathbb{E}_{\mu_{B_p^n}}[I_{B_p^n}]
&\geq \alpha_p n\int_0^1nr^{n-1}\log\frac1{1-r}\,dr -n\log A_p.
\end{align*}
The integral is comparable to $\log n$; more precisely,
$$\int_0^1nr^{n-1}\log\frac1{1-r}\,dr=H_n.$$
Hence
$$\mathbb{E}_{\mu_{B_p^n}}[I_{B_p^n}]\geq\alpha_p nH_n-n\log A_p\geq C_p n\log n-c_p n.$$
This proves the proposition.
\end{proof}

Combining Proposition~\ref{prop:p-ball-depth-mean} with Corollary~\ref{cor:sharp-threshold-tukey}, we obtain the desired sharp
threshold for random convex hulls generated by the uniform measure on $\ell_p$ balls.

\begin{corollary}\label{cor:sharp-threshold-p-balls}
For every $p>1$, the sequence of uniform probability measures on $B_p^n$ exhibits a sharp threshold around $\bigl( \mathbb{E}_{\mu_{B_p^n}} [\Lambda^\ast_{B_p^n}]\bigr)_{n \in \mathbb{N}}$.
\end{corollary}

\bigskip

\noindent {\bf Acknowledgements.} I would like to thank S.~Brazitikos and A.~Giannopoulos for helpful discussions.

\bigskip 

\medskip

\thanks{\noindent {\bf Keywords:} log-concave probability measures, half-space depth, Cram\'{e}r transform, random polytopes, convex bodies.}

\smallskip

\thanks{\noindent {\bf 2020 MSC:} Primary 60D05; Secondary 60E15, 62H05, 52A22, 52A23.}

\bigskip

\bigskip 

\bigskip

\noindent \textsc{Minas \ Pafis}: Department of Mathematics, National and Kapodistrian University of Athens, Panepistimioupolis 157-84,
Athens, Greece.

\smallskip

\noindent \textit{E-mail:} \texttt{mipafis@math.uoa.gr}

\end{document}